\documentclass[11pt, a4paper]{amsart}
\usepackage{mathpazo}
\usepackage{a4wide}
\usepackage{mathrsfs}
\usepackage{amssymb, amsmath}
\usepackage{enumitem}

\newcommand{\RR}{\mathbf{R}}

\newcommand{\teta}{\vartheta}

\newcommand{\ellal}{\ell^\infty_\mathscr{A}}

\newcommand{\Aut}{\mathrm{Aut}}
\newcommand{\path}{P}
\newcommand{\pathmod}{\RR_\mathrm{alt}[P]}
\newcommand{\pathban}{\mathscr{P}}
\newcommand{\pathnorm}[1]{\|#1\|_\mathrm{path}}
\newcommand{\coh}[1]{V_\mathrm{coh}^{#1}}
\newcommand{\hbc}{\mathrm{H}_\mathrm{cb}}
\newcommand{\hb}{\mathrm{H}_\mathrm{b}}

\newcommand{\se}{\subseteq}
\newcommand{\ol}{\overline}
\newcommand{\potimes}{\mathbin{\widehat\otimes}}
\newcommand{\cp}{\mathbin{\smallsmile}}
\newcommand{\inv}{^{-1}}

\newcommand{\car}{\sqsubset}
\theoremstyle{plain}
\newtheorem{thm}{Theorem}
\newtheorem{lem}[thm]{Lemma}
\newtheorem{prop}[thm]{Proposition}
\newtheorem{cor}[thm]{Corollary}

\begin{document}

\title[The cup product of quasimorphisms]{The cup product of Brooks quasimorphisms}

\begin{abstract}
We prove the vanishing of the cup product of the bounded cohomology classes associated to any two Brooks quasimorphisms on the free group. This is a consequence of the vanishing of the square of a universal class for tree automorphism groups.
\end{abstract}

\author[M. Bucher]{Michelle Bucher}
\address{University of Geneva, 1211 Geneva 4, Switzerland}
\email{michelle.bucher-karlsson@unige.ch}
\author[N. Monod]{Nicolas Monod}
\address{EPFL, SB-MATH-EGG, 1015 Lausanne, Switzerland}
\email{nicolas.monod@epfl.ch}
\maketitle

\section{Introduction}
Although bounded cohomology found a great variety of applications, it remains so mysterious that even for a (non-abelian) free group $F$ of finite rank, we do not know much about it. 

More precisely, beyond the trivial case of $\hb^1(F, \RR)=0$, it is known that both $\hb^2(F,\RR)$ and $\hb^3(F,\RR)$ are infinite-dimensional. However, $\hb^{n\geq 4}(F,\RR)$ remains completely unknown; in particular, we do not know whether $\hb^4(F,\RR)$ vanishes or not.

\medskip

The first infinite family of non-trivial classes in $\hb^2(F,\RR)$ are provided by \textbf{Brooks quasimorphisms}~\cite{Brooks} (anticipated by Johnson~\cite[2.8]{Johnson} and Rhemtulla~\cite{Rhemtulla68}); we recall their definition. Pick any reduced word $w$ in a choice of free generators for $F$ and consider the counting function $f_w\colon F\rightarrow \RR$ defined on $g\in F$ by
$$f_w(g)=\sharp\{\text{occurrences of $w$ in $g$}\} - \sharp\{\text{occurrences of $w$ in $g\inv$}\}.$$
If $w$ is reduced to one letter (or trivial), then $f_w$ is a homomorphism. In all other cases, $f_w$ is a quasimorphism and defines a non-trivial class $\beta_w\in \hb^2(F,\RR)$ unless $w$ is conjugated to a power of a letter. The space spanned by all these $\beta_w$ is infinite-dimensional~\cite{Brooks}\cite{Mitsumatsu} and is dense in $\hb^2(F,\RR)$ for a suitable topology of pointwise convergence~\cite[5.7]{Grigorchuk95}. (Following Brooks, we allow overlaps when counting occurrences, whilst other authors do not; see~\cite[p.~251]{Hartnick-Schweitzer} for the density in our setting.)

\smallskip

The aim of this note is to show that the cup product of any two elements in this dense sub-space vanishes in $\hb^4(F,\RR)$.

\begin{thm}\label{thm:cupOnF} 
Let $\beta_w,\beta_{w'}\in \hb^2(F,\RR)$ be the bounded cohomology classes associated to two Brooks quasimorphisms on $F$.

Then $\beta_w\cp \beta_{w'}=0$ in $\hb^4(F,\RR)$.
\end{thm}

We were informed by N.~Heuer that he independently obtained a similar result~\cite{Heuer_arx} by methods completely different from ours.

\medskip

We can give a rather transparent proof of Theorem~\ref{thm:cupOnF} by realizing bounded cohomology with the \emph{aligned chains} that we introduced in~\cite{Bucher-Monod_tree_inpress}. This simplifies the combinatorics and allows us to exhibit a natural explicit coboundary for the cup product.

Moreover, we can carry out this task at once for all $w, w'$ simultaneously ---~by working instead with the \textbf{universal class} $[\omega]$ that we now proceed to define (similar constructions were considered in~\cite[\S2]{Monod-ShalomCRAS}, in~\cite[7.11]{Monod-Shalom1} and in~\cite[\S9]{Duchesne-Monod_dendritesARX}).

\medskip

Let $T=(V, E)$ be a locally finite tree with Serre's conventions, which means in particular that an element of $E$ represents an \emph{oriented} edge and that $E$ is endowed with a fixed-point-free involution $e\mapsto\overline e$ reversing the orientation. We denote by $\path$ the set of \textbf{paths}, namely sequences $p=(e_1, \ldots, e_n)$ of successive edges $e_i\in E$ without backtracking. The \textbf{reverse} path is $\overline p = (\overline e_n, \ldots, \overline e_1)$ and $n$ is the \textbf{length} of $p$. Given two vertices $x,y$ we denote by $[x,y]$ the path connecting them. The \textbf{path module} $\pathmod$ is the $\Aut(T)$-module of all elements of the free vector space $\RR[\path]$ that change sign when replacing a path by its reverse. We define an $\Aut(T)$-equivariant  map $\teta\colon V^2\to\pathmod$ by setting
$$\teta(x_0, x_1)(p)= \pm 1$$
if $p$ (respectively $\ol p$) is contained as a sub-path in $[x_0, x_1]$, and~$0$ in all other cases. We define
$$\omega=d\teta\colon V^3\to\pathmod$$
as the coboundary of $\teta$. We recall here that $d$ will always be the usual alternating sum of the maps omitting the individual variables; we refer to the preliminaries below for explicit values of $\omega$.

\medskip

In order to view $\omega$ as a cocycle in bounded cohomology, we need to specify a norm on $\pathmod$; of course, $\teta$ should be unbounded for this norm since otherwise the class of $\omega$ would be trivial. The specific norm is however not too relevant; one property we want is that, when restricted to the free vector space on the set of paths of length $n$, it is equivalent to the $\ell^1$-norm $\|\cdot\|_{n,1}$. One explicit choice is the norm $\pathnorm\cdot = {\sum_{n\geq 1} \frac1{n!} \|\cdot\|_{n,1}}$ whose normalisation factor $1/n!$ is an arbitrary way to ensure uniform boundedness statements in the proofs.

Furthermore, we write $\pathban$ for the completion of $\pathmod$. Indeed, even though  our arguments will be explicit and finitary, the general tools of continuous bounded cohomology work best with \emph{Banach} spaces.

\medskip

A choice of free generators for the free group $F$ determines an embedding of $F$ into the automorphism group of the corresponding tree $T$. We view $\omega$ as a cocycle for the continuous bounded cohomology $\hbc^*$ of the locally compact group $\Aut(T)$.

Moreover, every path in $T$ is labelled by a reduced word in $F$. Thus, given a reduced word $w$, we can define an $F$-invariant bounded linear form $\lambda_w$ on $\pathmod$, hence also on $\pathban$, by specifying its values on individual paths as follows:
$$\lambda_w(p)=\left\{ \begin{array}{ll}
\phantom{-}1& \text{if $w$ labels $p$},\\
-1 & \text{if $w$ labels $\ol p$},\\
\phantom{-}0&\text{otherwise}. 
\end{array}\right.$$
This definition ensures that if $g\in F$ labels $[x_0, x_1]$, then
$$\lambda_w\circ\teta(x_0, x_1)= f_w(g).$$
Therefore, we deduce immediately the following relation between the universal class $[\omega]$ and individual quasimorphisms.

\begin{prop}\label{prop:mother}
Let $\beta_w\in\hb^2(F, \RR)$ be the bounded cohomology class associated to a Brooks quasimorphism on $F$ for the chosen generators. Then $\beta_w$ is the image of the class of $\omega$ under the map
$$\hbc^2\big(\Aut(T), \pathban\big) \xrightarrow{\ \mathrm{rest}\ } \hb^2(F, \pathban) \xrightarrow{\ (\lambda_w)_*\ } \hb^2(F, \RR),$$
where the first arrow is the restriction map and the second is induced by $\lambda_w$.\qed
\end{prop}

The cup product of two elements of $\hbc^2(\Aut(T), \pathban)$ is a class in $\hbc^4$ with values in the tensor product module $\pathban\otimes\pathban$, which we can also (projectively) complete to $\pathban\potimes\pathban$ (see the preliminaries for the norm). The naturality of the cup product now implies:

\begin{cor}\label{cor:mother}
Given two reduced words $w$ and $w'$, we keep all the above notation.

Then $[\omega]\cp[\omega]$, viewed as a class with coefficients in $\pathban\potimes\pathban$, is mapped to $\beta_w\cp\beta_{w'}$
$$\hbc^4\big(\Aut(T), \pathban\potimes\pathban\big) \longrightarrow \hb^4(F, \RR)$$
under the restriction followed by $(\lambda_w\otimes\lambda_{w'})_*$.\qed
\end{cor}

In view of Corollary~\ref{cor:mother}, Theorem~\ref{thm:cupOnF} is now an immediate consequence of the following vanishing result for the square of the universal class $[\omega]$.

\begin{thm}\label{mainthm}
Let $T$ be a locally finite tree.

Then the class of $\omega\cp \omega$ vanishes in $\hbc^4\big(\Aut(T), \pathban\potimes\pathban\big)$.
\end{thm}

\noindent
The remainder of this note is devoted to the proof of Theorem~\ref{mainthm}.

\section{Preliminaries}
The cup-product of bounded cocycles ranging in $\pathmod$ or $\pathban$ is trivially bounded for \emph{any} cross-norm on the tensor product with underlying norm $\pathnorm\cdot$ on each factor. For cross-norms, we refer to~\cite{Ryan_book}. We shall choose the \textbf{projective} cross-norm $\|\cdot\|_\pi$ and denote by $\pathban\potimes\pathban$ the corresponding completion. Since this is the largest cross-norm, the vanishing result of Theorem~\ref{mainthm} with respect to $\|\cdot\|_\pi$ implies the corresponding vanishing for any other cross-norm.

\medskip

We say that a path $p$ is \textbf{carried} by a path $q$, and write $p \car q$, if either $p$ or $\ol p$ is contained in $q$ as a sub-path. We attach a sign $\pm 1$ to these two cases, referred to as the \textbf{orientation} of $p$ relative to $q$. We define the \textbf{interior} $\mathrm{Int}(p)\se V$ of a path $p$ to consist of all the vertices of the path except its two extremities. 

Recall that any three vertices $x_0, x_1, x_2\in V$ determine a \textbf{center} $c\in V$ characterized as the unique common vertex of all $[x_i, x_j]$. Given a path $p$, the definition of $\omega$ now shows that $\omega(x)(p)=\pm 1$ when $p$ is carried by some $[x_i, x_j]$ and $c\in \mathrm{Int}(p)$, and that $\omega(x)(p)$ vanishes otherwise.

A path can contain at most $n-1$ sub-paths of length $n$ containing a given vertex in their interior. Therefore, considering all three configurations and two orientations, we can bound the norm of $\omega$ by
$$\pathnorm{\omega(x)} \leq 3\cdot 2\cdot \sum_{n\geq 1} \frac1{n!} (n-1) =6,$$
witnessing that $\omega$ is indeed uniformly bounded.

\medskip

Recall that a $(q+1)$-tuple $(x_0,\dots,x_q)\in V^{q+1}$ is \textbf{aligned} if the vertices $x_0,\dots,x_q$ are contained in some geodesic segment of $T$. This tuple is furthermore said to be \textbf{coherent} if these $q+1$ vertices are distinct and in increasing order for one of the two linear orders induced on $\{x_0,\dots,x_q\}$ by any such segment. We denote by $\coh{q+1} \se V^{q+1}$ the set of coherent aligned tuples.

Below, we shall be particularly interested in the above description of $\omega(x)$ specialized to coherent triples $x\in \coh3$. In that case, $\omega(x)(p)=\pm 1$ if $x_1\in \mathrm{Int}(p)$ and $p$ is carried by $[x_0,x_2]$, with the sign given by the orientation of $p$ relative to $[x_0, x_2]$, and vanishes in all other cases.
\begin{center}
\setlength{\unitlength}{0.7cm}
\thicklines
\begin{picture}(9,3)
\put(6.5,1) {$\Longrightarrow\kern2mm \omega=1.$}
\put(1,1){\line(1,0){4}}
\put(1,1){\circle*{0.2}}\put(0.9,0.3){$x_0$}
\put(3,1){\circle*{0.2}}\put(2.9,0.3){$x_1$}
\put(5,1){\circle*{0.2}}\put(4.9,0.3){$x_2$}
\put(2,1.5){\vector(1,0){1.9}}\put(2.9,1.8){$p$}
\end{picture}
\end{center}

\section{The coherent resolution}
Let $E$ be any isometric Banach $\Aut(T)$-module and recall that $\hbc^q(\Aut(T), E)$ can be computed with the (non-augmented) complex $\ell^\infty(V^{q+1}, E)^{\Aut(T)}$ of $\Aut(T)$-equivariant elements of the resolution
\begin{equation}\label{eq:homog}
0 \longrightarrow E \longrightarrow\ell^\infty(V, E) \longrightarrow \ell^\infty(V^2,E)  \longrightarrow \ell^\infty(V^3,E)  \longrightarrow \cdots
\end{equation}
(see e.g.~\cite[4.5.2]{Monod}). There is a natural restriction map to the complex $\ell^\infty(\coh{q+1}, E)$ on coherent tuples, but we warn the reader that \emph{the latter is not a resolution} of $E$.

Recall that an element of $\ell^\infty(V^{q+1}, E)$ is called \textbf{alternating} if any permutation $\sigma$ of the variables corresponds to the multiplication by the signature $\mathrm{sign}(\sigma)$. We denote by $\tau_q$ the permutation of $\{0, \ldots, q\}$ that reverses the order and observe that its signature is $(-1)^{\lfloor \frac{q+1}{2}\rfloor}$, where ${\lfloor \cdot\rfloor}$ denotes the integer part. Consider the $\Aut(T)$-equivariant involution $\hat\tau_q$ of $\ell^\infty(\coh{q+1}, E)$ defined by $\hat\tau_q (\alpha) (x) = \mathrm{sign}(\tau_q )\alpha(x^{\tau_q})$. Being an involution, it induces an eigenspace decomposition
$$\ell^\infty(\coh{q+1}, E) = \ell^\infty_+(\coh{q+1}, E) \oplus \ell^\infty_-(\coh{q+1}, E)$$
which is preserved by $\Aut(T)$. Although  $\ell^\infty(\coh{q+1}, E)$ is not a resolution, we have:

\begin{prop}\label{prop:comp}
The sub-complex
\begin{equation}\label{eq:coh}
0 \longrightarrow E \longrightarrow \ell^\infty_+(\coh 1, E) \longrightarrow \ell^\infty_+(\coh 2, E) \longrightarrow \ell^\infty_+(\coh 3, E) \longrightarrow \cdots
\end{equation}
is a resolution. Moreover, the map
$$A_q\circ \mathrm{rest} \colon \ell^\infty(V^{q+1}, E) \longrightarrow  \ell^\infty_+(\coh{q+1}, E)$$
from~\eqref{eq:homog} to~\eqref{eq:coh} obtained by restriction followed by the projection $A_q=(\hat\tau_q+\mathrm{Id})/2$ yields an isomorphism between $\hbc^q(\Aut(T), E)$ and the cohomology of the complex
\begin{equation}\label{eq:coh:inv}
0  \longrightarrow \ell^\infty_+(\coh 1, E)^{\Aut(T)} \longrightarrow \ell^\infty_+(\coh 2, E)^{\Aut(T)} \longrightarrow \ell^\infty_+(\coh 3, E)^{\Aut(T)} \longrightarrow \cdots
\end{equation}
\end{prop}

\begin{proof}
Following~\cite{Bucher-Monod_tree_inpress}, we denote by $\ellal(V^{q+1}, E)$ the sub-space of alternating maps defined on aligned tuples. The restriction to coherent tuples thus induces an isomorphism
$$\ellal(V^{q+1}, E) \cong \ell^\infty_+(\coh{q+1}, E).$$
Therefore, the first statement is simply a reformulation of Corollary~8 of~\cite{Bucher-Monod_tree_inpress}. Moreover, as observed there, the modules $\ellal(V^{q+1}, E)$ are relatively injective in the sense of bounded cohomology because the $\Aut(T)$-action on the set of aligned tuples is proper, see~\cite[4.5.2]{Monod}. More precisely, $\ellal(V^{q+1}, E)$ is a direct factor of the larger space without the alternation condition, to which~\cite[4.5.2]{Monod} applies, and one concludes as in~\cite[7.4.5]{Monod} by an alternation map.

A direct computation using the relation $ \mathrm{sign}(\tau_q ) \cdot \mathrm{sign}(\tau_{q+1}) = (-1)^{q+1}$ shows that $\hat\tau_q$ is a chain map. In particular, $\hat\tau_q$ automatically preserves the decomposition $\ell^\infty_\pm(\coh{q+1}, E)$ and $A_q$ is a chain map as well. Now the second statement follows by general cohomological principles (see e.g.\ \S7.2 in~\cite{Monod}).
\end{proof}

\section{A primitive for the square of $\omega$ on coherent tuples}
We define an $\Aut(T)$-equivariant map
$$B\colon \coh4 \longrightarrow \pathmod\otimes\pathmod$$
by setting, for any coherent $4$-tuple $x$ and any paths $p_1, p_2\in \path$,
$$B(x)(p_1,p_2)=\pm1$$
whenever all the following hold:
\begin{itemize}
\item  both $p_1$ and $p_2$ are carried by $[x_0,x_3]$,
\item the interior of $p_1$ and of $p_2$ are disjoint,
\item $x_i\in \mathrm{Int}(p_i)$ for each $i=1,2$.
\end{itemize}
In that case, the sign $\pm1$ is the product of the orientations of $p_1$ and of $p_2$ relative to $[x_0,x_3]$. All this is perhaps much more intuitive in a picture, drawn for two of the four orientation possibilities:
\begin{center}
\setlength{\unitlength}{0.7cm}
\thicklines
\begin{picture}(11,3)
\put(8.5,1) {$\Longrightarrow\kern2mm B=1.$}
\put(1,1){\line(1,0){6}}
\put(1,1){\circle*{0.2}}\put(0.9,0.3){$x_0$}
\put(3,1){\circle*{0.2}}\put(2.9,0.3){$x_1$}
\put(5,1){\circle*{0.2}}\put(4.9,0.3){$x_2$}
\put(7,1){\circle*{0.2}}\put(6.9,0.3){$x_3$}
%
\put(2,1.5){\vector(1,0){1.9}}\put(2.9,1.8){$p_1$}
\put(4.1,1.5){\vector(1,0){1.9}}\put(4.9,1.8){$p_2$}
\end{picture}
\end{center}
\begin{center}
\setlength{\unitlength}{0.7cm}
\thicklines
\begin{picture}(11,3)
\put(8.5,1) {$\Longrightarrow\kern2mm B=-1.$}
\put(1,1){\line(1,0){6}}
\put(1,1){\circle*{0.2}}\put(0.9,0.3){$x_0$}
\put(3,1){\circle*{0.2}}\put(2.9,0.3){$x_1$}
\put(5,1){\circle*{0.2}}\put(4.9,0.3){$x_2$}
\put(7,1){\circle*{0.2}}\put(6.9,0.3){$x_3$}
%
\put(2,1.5){\vector(1,0){1.9}}\put(2.9,1.8){$p_1$}
\put(6,1.5){\vector(-1,0){1.9}}\put(4.9,1.8){$p_2$}
\end{picture}
\end{center}
In all other cases, we set $B(x)(p_1,p_2)=0$.

\begin{prop}\label{prop:db}
We have $d B(x)=\omega\cp \omega(x)$ for every coherent $5$-tuple $x$.
\end{prop}

\begin{proof} Let $p_1, p_2\in \path$. By definition,
$$\omega\cp \omega(x)(p_1,p_2)=\omega(x_0,x_1,x_2)(p_1)\cdot \omega(x_2,x_3,x_4)(p_2).$$
Thus, $\omega\cp \omega(x)(p_1,p_2)\neq 0$ if and only if all the following hold:
\begin{equation}\left\{ \begin{array}{l} \label{condition1}
x_1\in \mathrm{Int}(p_1) \ \text{ and }\ p_1 \car [x_0,x_2],\\
x_3\in \mathrm{Int}(p_2)  \ \text{ and }\ p_2 \car [x_2,x_4].
\end{array} \right.\end{equation}
As for $d B$, we observe that $d B(x)(p_1,p_2)=0$ unless possibly
\begin{equation}\left\{ \begin{array}{l} \label{condition2}
p_1,p_2 \text{ have disjoint interior and are carried by } [x_0,x_4],\\
x_1 \mathrm{\ or \ } x_2 \in \mathrm{Int}(p_1),\\
x_2 \mathrm{\ or \ } x_3\in \mathrm{Int}(p_2).
\end{array} \right.\end{equation}
In the case when Conditions (\ref{condition2}) are not satisfied, Conditions (\ref{condition1}) are not either; therefore in that case $d B$ and $\omega\cp \omega$ agree since they both vanish.

\medskip
Suppose now that Conditions (\ref{condition2}) are satisfied. By symmetry, we can assume that the orientation of $p_1$ and $p_2$ are compatible with the orientation of $[x_0,x_4]$ (and hence of $[x_0,x_3]$ and of $[x_1,x_4]$). Since $p_1$ and $p_2$ have disjoint interior, $x_2$ is contained in at most one of $\mathrm{Int}(p_1)$ or $\mathrm{Int}(p_2)$; we can suppose that it is not contained in $\mathrm{Int}(p_1)$, the other case being completely analogous. We have now three cases:

\medskip
\textbf{First case:} $x_1\in \mathrm{Int}(p_1)$, $x_2\notin \mathrm{Int}(p_1)\cup \mathrm{Int}(p_2)$ and $x_3\in \mathrm{Int}(p_2)$.
\begin{center}
\setlength{\unitlength}{0.7cm}
\thicklines
\begin{picture}(11,3)
\put(1,1){\line(1,0){8}}
\put(1,1){\circle*{0.2}}\put(0.9,0.3){$x_0$}
\put(3,1){\circle*{0.2}}\put(2.9,0.3){$x_1$}
\put(5,1){\circle*{0.2}}\put(4.9,0.3){$x_2$}
\put(7,1){\circle*{0.2}}\put(6.9,0.3){$x_3$}
\put(9,1){\circle*{0.2}}\put(8.9,0.3){$x_4$}
\put(2,1.5){\vector(1,0){2}}\put(2.9,1.8){$p_1$}
\put(6,1.5){\vector(1,0){2}}\put(6.9,1.8){$p_2$}
\end{picture}
\end{center}
The value of $\omega\cp \omega(x)(p_1,p_2)$ is $+1$, while the only non-zero summand in
$$d B(x)(p_1,p_2)=\Sigma_{i=0}^4 (-1)^i  B(\dots,\widehat{x_i},\dots )$$
is the one for $i=2$, which is indeed also $+1$.

\medskip
\textbf{Second case:} $x_1\in \mathrm{Int}(p_1)$ and $x_2, x_3\in \mathrm{Int}(p_2)$.
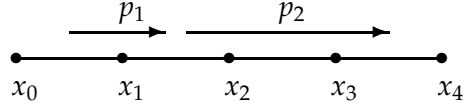
\begin{center}
\setlength{\unitlength}{0.7cm}
\thicklines
\begin{picture}(11,3)
\put(1,1){\line(1,0){8}}
\put(1,1){\circle*{0.2}}\put(0.9,0.3){$x_0$}
\put(3,1){\circle*{0.2}}\put(2.9,0.3){$x_1$}
\put(5,1){\circle*{0.2}}\put(4.9,0.3){$x_2$}
\put(7,1){\circle*{0.2}}\put(6.9,0.3){$x_3$}
\put(9,1){\circle*{0.2}}\put(8.9,0.3){$x_4$}
\put(2,1.5){\vector(1,0){1.8}}\put(2.9,1.8){$p_1$}
\put(4.2,1.5){\vector(1,0){3.8}}\put(5.9,1.8){$p_2$}
\end{picture}
\end{center}
Condition (\ref{condition1}) is not satisfied and hence $\omega\cp \omega$ vanishes. As for $d B$, only the summands for $i=2$ and $i=3$ are non-zero and cancel out to give $d B(x)(p_1,p_2)=0$.

\medskip
\textbf{Third case:} $x_1\in \mathrm{Int}(p_1)$, $x_2\in \mathrm{Int}(p_2)$ and $x_3\notin \mathrm{Int}(p_2)$.
\begin{center}
\setlength{\unitlength}{0.7cm}
\thicklines
\begin{picture}(11,3)
\put(1,1){\line(1,0){8}}
\put(1,1){\circle*{0.2}}\put(0.9,0.3){$x_0$}
\put(3,1){\circle*{0.2}}\put(2.9,0.3){$x_1$}
\put(5,1){\circle*{0.2}}\put(4.9,0.3){$x_2$}
\put(7,1){\circle*{0.2}}\put(6.9,0.3){$x_3$}
\put(9,1){\circle*{0.2}}\put(8.9,0.3){$x_4$}
\put(2,1.5){\vector(1,0){1.8}}\put(2.9,1.8){$p_1$}
\put(4.2,1.5){\vector(1,0){1.8}}\put(4.9,1.8){$p_2$}
\end{picture}
\end{center}
Again, condition (\ref{condition2}) is not satisfied and $\omega\cp \omega$ vanishes. As for $d B$, only the summands for $i=3$ and $i=4$ are non-zero and cancel out to give $d B(x)(p_1,p_2)=0$.
\end{proof}

\section{Proof of Theorem~\ref{mainthm}}
We first verify that the primitive $B$ is bounded.

\begin{lem}\label{lem:B:bnd}
The map $B$ is uniformly bounded on $\coh4$ with respect to the projective norm $\|\cdot\|_\pi$.
\end{lem}

\begin{proof}
Fix $x\in \coh4$ and consider abusively any path $p_i$ as an element of $\RR[P]$. By definition of the projective cross-norm, we can bound $\|B(x) \|_\pi$ by $\sum (\pathnorm{p_1} \cdot \pathnorm{p_2} )$, where the sum runs over all pairs $(p_1, p_2)$ on which $B(x)$ does not vanish. Arguing as in our estimate for the norm of $\omega$, we have at most $2(n_1-1)(n_2-1)$ such pairs whenever we fix the length $n_i$ of each $p_i$.
Since on the other hand we have $\pathnorm{p_i}=1/n_i!$, we conclude that $B(x)$ has norm at most
\[\sum_{n_1, n_2} \frac{2(n_1-1)(n_2-1)}{n_1! n_2!} = 2 \left(\sum_n \frac{n-1}{n!}\right)^2 =2.\]
\end{proof}

At this point, we conclude that $A_3 (B)$ belongs to $\ell^\infty_+(\coh4, \pathban)$. Since $A_*$ is a chain map (as pointed out in the proof of Proposition~\ref{prop:comp}), we deduce from Proposition~\ref{prop:db} that we have $A_4 (\omega\cp\omega) = d A_3 (B)$. Now Proposition~\ref{prop:comp} implies that the class of $\omega\cp\omega$ vanishes, completing the proof of Theorem~\ref{mainthm}.\qed


\medskip
\noindent
\textbf{Acknowledgements.}
The authors are grateful to Tobias Hartnick for his comments.



\bibliographystyle{../BIB/amsplain}
\bibliography{../BIB/ma_bib}

\end{document}